\newif\ifalggeom\alggeomfalse
\ifalggeom
  \documentclass{alggeom}
\else
  \documentclass[11pt,reqno]{amsart}
\fi
\usepackage{amsthm,amsmath,amssymb}
\usepackage{hyperref}
\usepackage{empheq}
\usepackage{xcolor}
\usepackage{mathrsfs}
\usepackage{url}
\usepackage{fullpage}
\usepackage{enumerate}
\usepackage{stmaryrd}
\usepackage[all]{xy}
  \SelectTips{cm}{10}
  \everyxy={<2.5em,0em>:}

\theoremstyle{plain}
  \newtheorem{theorem}{Theorem}[section]
  \newtheorem{lemma}[theorem]{Lemma}
  \newtheorem{proposition}[theorem]{Proposition}
  
  \newtheorem{conjecture}[theorem]{Conjecture}
\theoremstyle{definition}
  \newtheorem{definition}[theorem]{Definition}
  \newtheorem*{definition*}{Definition}
  \ifalggeom\else\newtheorem*{acknowledgements}{Acknowledgements}\fi
\theoremstyle{remark}
  \newtheorem{remark}[theorem]{Remark}
  \newtheorem*{remark*}{Remark}
  \newtheorem{question}[theorem]{Question}
  \newtheorem{eggSample}[theorem]{Example}
    \makeatletter
      \newenvironment{example}[1][]{
        \begin{eggSample}[#1]}
        {\hfill$\diamond$\end{eggSample}}
    \makeatother

\newcommand\notocsubsection[1]{\medskip\noindent\textbf{#1.}}

 \DeclareFontFamily{OT1}{pzc}{}                                 
 \DeclareFontShape{OT1}{pzc}{m}{it}{<-> s * [1.100] pzcmi7t}{}  
 \DeclareMathAlphabet{\mathpzc}{OT1}{pzc}{m}{it}                

\newcommand{\loc}{\mathit{loc}}

\newcommand{\Q}{\mathcal{Q}}
\DeclareMathOperator{\Hilb}{\underline{HS}}
\newcommand{\h}{h}
\DeclareMathOperator{\Br}{Br}
\newcommand{\tor}{\mathrm{tor}}
\DeclareMathOperator{\free}{free}
\DeclareMathOperator{\Lie}{Lie}

\renewcommand{\AA}{\mathbb{A}}
\newcommand{\CC}{\mathbb{C}}

\DeclareMathOperator{\Coh}{Coh}
\DeclareMathOperator{\qcoh}{QCoh}
\newcommand{\liset}{\text{\it lis-et}}

\newcommand{\m}{{\mathfrak{m}}}
\newcommand{\hhat}[1]{\widehat{#1}}
\newcommand{\hhatX}{\hspace{.7ex}\widehat{\phantom{\rule{2ex}{1.6ex}}}\hspace{-2.7ex}\X}
\newcommand{\ttilde}[1]{\widetilde{#1}}
\DeclareMathOperator{\spec}{Spec}
\DeclareMathOperator{\proj}{Proj}

\DeclareMathOperator{\Aut}{Aut}
\let\hom\relax \DeclareMathOperator{\hom}{Hom}
\DeclareMathOperator{\Hom}{\ensuremath{\mathscr{H}\kern-.35em\mathpzc{om}}}

\newcommand{\X}{\mathscr{X}}

\newcommand{\Y}{\mathscr{Y}}
\newcommand{\I}{\mathcal I}
\newcommand{\F}{\mathcal F}
\newcommand{\V}{\mathcal V}
\newcommand{\W}{\mathcal W}
\newcommand{\fF}{\mathfrak F}
\newcommand{\fG}{\mathfrak G}

\renewcommand{\O}{\mathcal O}

\newcommand{\ZZ}{\mathbb{Z}}
\newcommand{\GG}{\mathbb{G}}
\newcommand{\G}{\mathcal G}
\newcommand{\fM}{\mathfrak{M}}

\newcommand{\llparen}{\ensuremath{\mbox{$(\!($}}}
\newcommand{\rrparen}{\ensuremath{\mbox{$)\!)$}}}
\renewcommand{\mod}{\mathrm{\text{-}mod}}

\newcommand{\setlabel}[1]{\phantomsection\label{#1}{\rm (#1)}}
\newcommand{\myref}[1]{\hyperref[#1]{\rm (#1)}}

\begin{document}
\title{Formal GAGA for Good Moduli Spaces}
\author{Anton Geraschenko}
\address{Google, 340 Main St., Venice, CA 90291 USA}
\email{geraschenko@gmail.com}
\author{David Zureick-Brown}
\address{Department of Mathematics and Computer Science, Emory University, Atlanta, GA 30322 USA}
\email{dzb@mathcs.emory.edu}
\date{}

\begin{abstract}
  We prove formal GAGA for good moduli space morphisms under an assumption of ``enough   vector bundles'' (which holds for instance for quotient stacks). This supports the philosophy that though they are non-separated, good moduli space morphisms largely behave like proper morphisms.
\end{abstract}
\maketitle

\section{Introduction}

Good moduli space morphisms are a common generalization of good quotients by linearly reductive group schemes \cite{git} and coarse moduli spaces of tame Artin stacks \cite[Definition 3.1]{AbramovichOV:Tame}.
\begin{definition*}[{\cite[Definition 4.1]{Alper:good}}]
 A quasi-compact and quasi-separated morphism of locally Noetherian algebraic stacks $\phi\colon \X\to \Y$ is a \emph{good moduli space morphism} if
 \begin{itemize}
  \item ($\phi$ is \emph{Stein}) the morphism $\O_\Y\to \phi_*\O_\X$ is an isomorphism, and
  \item ($\phi$ is \emph{cohomologically affine}) the functor $\phi_*\colon \qcoh(\O_\X)\to \qcoh(\O_\Y)$ is exact.
 \end{itemize}
\end{definition*}
If $\phi\colon \X \to \Y$ is such a morphism, then any morphism from $\X$ to an algebraic space factors through $\phi$ \cite[Theorem 6.6]{Alper:good}.\footnote{If $\Y$ is an algebraic space then this is \cite[Theorem 6.6]{Alper:good}.
More generally, since algebraic spaces are sheaves in the smooth topology, this property may be checked smooth locally on $\Y$, and since
good moduli space morphisms are stable under base change \cite[Proposition 4.7(i)]{Alper:good}, this follows from the case of $\Y$ an algebraic space.} 
In particular, if there exists a good moduli space morphism $\phi\colon \X\to X$ where $X$ is an algebraic space, then $X$ is determined up to unique isomorphism. In this case, $X$ is said to be \emph{the good moduli space} of $\X$. If $\X=[U/G]$, this corresponds to $X$ being a good quotient of $U$ by $G$ in the sense of \cite{git} (e.g.~for a linearly reductive $G$, $[\spec R/G]\to \spec R^G$ is a good moduli space).

In many respects, good moduli space morphisms behave like proper morphisms. They are universally closed \cite[Theorem 4.16(ii)]{Alper:good} and weakly separated \cite[Proposition 2.17]{ASW:weakly-proper}, but since points of $\X$ can have non-proper stabilizer groups, good moduli space morphisms are generally not separated (e.g.~if $G$ is a non-proper group scheme, $BG$ is not separated). Pushforward along a good moduli space morphism respects coherence \cite[Theorem 4.16(x)]{Alper:good}. 

The main theorem in this paper continues this philosophy, showing that formal GAGA holds for good moduli space morphisms, at least when the stack has ``enough vector bundles.'' Recall that a stack is said to have \emph{the resolution property} if every coherent sheaf has a surjection from a vector bundle. Recall also that if $\X\to X$ is a good moduli space morphism and $X$ has a unique closed point, then $\X$ also has a unique closed point \cite[Theorem 4.16(iii) and Proposition 9.1]{Alper:good}.

\begin{theorem}\label{T:main}
 \renewcommand{\leftmarginii}{4em}
 Suppose $\X\to \spec A$ is a good moduli space, where $A$ is a complete Noetherian local ring with maximal ideal $\m$ and $\X$ is of finite type over $\spec A$. Let $\hhatX$ denote the formal completion of $\X$ with respect to $\m$ (see \S\ref{S:terminology}).
 \begin{enumerate}[i.]
   \item\label{item:full-faithfulness} The completion functor $\Coh(\X)\to \Coh(\hhatX)$ is fully faithful.
   \item\label{item:existence} Suppose $\X_0=\X\times_{\spec A}\spec A/\m$ has the resolution property (e.g.~$\X_0$ is a quotient stack; see Remark \ref{R:affine-gms+quot=>res}). Then the following conditions are equivalent:
   \begin{enumerate}
    \item[\setlabel{quot}] $\X$ is the quotient of an affine scheme by $GL_n$ for some $n$.
    \item[\setlabel{quot$'$}] $\X$ is the quotient of an algebraic space by an affine algebraic group.
   \end{enumerate}
   The above conditions imply the following equivalent conditions: 
   \begin{enumerate}
    \item[\setlabel{res}] $\X$ has the resolution property.
    \item[\setlabel{res$'$}] Every coherent sheaf on $\X_0$ has a surjection from a vector bundle on $\X$.
   \end{enumerate}
   The above conditions imply
   \begin{enumerate}
    \item[\setlabel{GAGA}] The completion functor $\Coh(\X)\to \Coh(\hhatX)$ is an equivalence.
   \end{enumerate}
   If the unique closed point of $\X$ has affine stabilizer group then \myref{res} implies \myref{quot$'$}, and if $\X$ has affine diagonal then \myref{GAGA} implies \myref{res$'$}.
 \end{enumerate}
\end{theorem}

We provide examples in \S\ref{S:counterexamples} to show that \myref{GAGA} may fail under weaker hypotheses.

\begin{remark}
\label{R:AHR}
  As this paper went to press, we learned of a forthcoming result by Jarod Alper, Jack Hall, and David Rydh that implies many stacks satisfy \myref{quot} after an \'etale base change on their good moduli spaces. Combined with our Theorems \ref{T:main} and \ref{T:etale-local}, it implies that if $\X$ has affine diagonal, all the conditions in Theorem \ref{T:main} hold (see Remark \ref{R:affine-diag=>GAGA} and Conjecture \ref{conjecture}).
\end{remark}

\begin{remark}
  In \cite[Theorem 1.4]{MO:proper} (see also \cite{BC:gaga}), Olsson proves that formal GAGA holds for \emph{proper}
  Artin stacks. His main theorem gives a proper surjection from a proper scheme $X \to \X$, and formal GAGA   follows from a d\'evissage (as outlined in   \cite[\S 1.2]{HallR:hilbertStack}). In our setting such a surjection does   not exist, and our arguments are quite different.
\end{remark}

\begin{remark}
  If $\X$ has quasi-finite diagonal over a base $S$, the Hilbert stack $\Hilb_{\X/S}$ of quasi-finite representable $S$-maps with domain a proper $S$-stack is an algebraic stack \cite[Theorem 2]{HallR:hilbertStack}.   A key ingredient in the proof of this result is a weaker variant of formal GAGA for non-separated stacks.
\end{remark}

\begin{remark}\label{R:template}
  Formal GAGA allows the study of a stack $\X$ with good   moduli space $X$ to be largely reduced to the study of the fibers of the map   $\X\to X$. This reduction is particularly appealing since it is possible that the geometric fibers of this map must be quotient stacks (see Question \ref{Q:non-quotient_with_one_point_GMS} and Remark \ref{R:Jarod-conj=>question}). Here is the template for the reduction:
  \begin{enumerate}
  \item[(0)] \emph{Start with a problem which is \'etale local on $X$, and a solution to the problem for the fiber over a point $x$.}
  \item \emph{Use deformation theory to extend the solution to a formal solution}. Deformation theory typically shows that the problem of extending a solution from an infinitesimal neighborhood to a larger infinitesimal neighborhood is controlled by the cohomologies of certain quasi-coherent sheaves. If all higher cohomology groups of $\X\to \spec A$ vanish (e.g.~if $\X$ has affine diagonal; see Remark \ref{R:cohomological-dimension-error}), then deformation-theoretic problems are more or less trivial when working with good moduli space morphisms. (See Lemma \ref{L:vectorUnicity} as an example of this.)
  \item \emph{Show that any formal solution is effectivizable}. That is, show that any compatible family of solutions over all infinitesimal neighborhoods of $x\in X$ is induced by a solution over $\spec \hhat\O_{X,x}$. If the question can be formulated entirely in terms of coherent sheaves, as is often the case, then \myref{GAGA} does this step.
  \item \emph{Use Artin approximation \cite[Theorem 1.12]{artin:approximation} to extend the solution to an \'etale neighborhood of $x$.} If the stack of solutions is locally finitely presented, Artin's theorem  says that for a map $f$ from the complete local ring at a point, there is a map from the henselization of the local ring which agrees with $f$ modulo any given power of the maximal ideal. (By \cite[Proposition 2.3.8]{lieblichOsserman:functorialReconstruction}, one can instead apply Artin's theorem to the associated functor of isomorphism classes.) By step 1 (uniqueness of deformations) and formal GAGA, this must actually be an extension of $f$. By local finite presentation, this map extends to some \'etale neighborhood, as the henselization is the limit of all \'etale neighborhoods.
  \end{enumerate}
  Proposition \ref{P:localQuotient} illustrates this template. It shows that if $\X \to X$ is a good moduli space, $x \in X$ is a point at which formal GAGA holds, and the fiber over $x$ is a quotient   stack, then there is some \'etale neighborhood of $x$ over which $\X$ is a quotient stack.
\end{remark}

\begin{remark}[Related work]
\label{R:related-work}

Previous work \cite[Theorem 1.7]{alexeevB:moduliOfAffineSchemes},  \cite[Theorem 7.6]{alexeevB:stableReductiveI}, and \cite[Theorem 2.20]{brion:invariantHilberg} proves that the Hilbert-scheme of $G$-equivariant multiplicity-finite subschemes of an affine scheme exists when $G$ is connected reductive and of characteristic 0.
Working over the spectrum $X$ of a complete local ring, let $V \to X$ be an affine morphism with an action of a linearly reductive group $G$. Then formal GAGA holds for flat closed substacks $\mathcal{Z}$ of $\X = [V/G]$ whose good moduli space $Z$ is finite over $X$. The characteristic $p$ case (with linearly reductive instead of reductive) follows similarly from the existence of the multigraded Hilbert scheme \cite{haimanS:multigradedHilbert}, since  there are few linearly reductive group schemes in characteristic $p$ -- any such scheme is the extension of a linear reductive finite flat group scheme $G$ by a torus, and \cite{AbramovichOV:Tame} classified all such $G$.

The present work is a natural and direct proof of formal GAGA, extending this previous work to non-flat substacks and to arbitrary coherent sheaves.  While we work with the more restrictive hypothesis that $Z = X$, David Rydh has pointed out that it is easy to modify our argument to allow for stacks with separated good moduli space of finite type over $X$ and closed substacks $\mathcal{Z}$ whose good moduli space $Z$ is proper over $X$ (and similarly for coherent sheaves). Finally, while our work allows for more general stacks, in the main interesting case where $\X$ has affine stabilizers, our main theorem gives that (GAGA) $\Leftrightarrow$ (res) $\Leftrightarrow$ $\X=[V/G]$ with $V$ affine and $G=GL_n$.

\end{remark}

\begin{acknowledgements}
  We thank Jarod Alper, who introduced us to the problem and generously provided valuable feedback throughout our work on it. We   also thank Jack Hall, Martin Olsson, David Rydh, and Matt Satriano for many helpful discussions. We thank
\href{http://mathoverflow.net/questions/727/algebraic-group-g-vs-algebraic-stack-bg/1332#1332}{David Ben-Zvi},
\href{http://mathoverflow.net/questions/1809/does-every-morphism-bg-bh-come-from-a-homomorphism-g-h/1819#1819}{Bhargav Bhatt},
\href{http://mathoverflow.net/questions/3911/constructing-a-degeneration-as-a-group-scheme-of-g-m-to-g-a/3918#3918}{Scott Carnahan},
Torsten Ekedahl,
\href{http://mathoverflow.net/questions/4062/can-hom-gpg-h-fail-to-be-representable-for-affine-algebraic-groups/4105#4105}{David} \href{http://mathoverflow.net/questions/2710/can-isomorphisms-of-schemes-be-constructed-on-formal-neighborhoods/2720#2720}{Speyer},
\href{http://mathoverflow.net/questions/56112/are-representations-of-a-linearly-reductive-group-discretely-parameterized/56159#56159}{Angelo} \href{http://mathoverflow.net/questions/56862/is-there-a-non-quotient-stack-with-affine-st}{Vistoli},
\href{http://mathoverflow.net/questions/727/algebraic-group-g-vs-algebraic-stack-bg/746#746}{Ben Webster}, and
\href{http://mathoverflow.net/questions/570/deformation-theory-of-representations-of-an-algebraic-group/3187#3187}{Jonathan Wise}
for useful discussions on \href{http://mathoverflow.net}{MathOverflow}.
We especially thank the anonymous referee, who provided extensive useful comments, and thank Jack Hall and David Rydh a second time for very detailed comments, and the correspondence about related work, which became Remark \ref{R:related-work}.
The second author was partially supported by a National Defense Science and Engineering Graduate Fellowship and by a National Security Agency Young Investigator grant.
\end{acknowledgements}

%

\tableofcontents


\section{Terminology}\label{S:terminology}
This paper follows the conventions of \cite{Alper:good}. In particular, all schemes are assumed to be quasi-separated, stacks have quasi-compact diagonal,  all morphisms of stacks are assumed to be quasi-compact and quasi-separated. We fix a base scheme $X$, which will often be isomorphic to $\spec A$ where $A$ is a complete Noetherian local ring with maximal ideal $\m$. An affine algebraic group is understood to be a flat (over $X$) subgroup scheme of $GL_n := GL_{n,X}$, and a quotient stack $[U/G]$ is the quotient of an algebraic space $U \to X$ by an affine algebraic group over $S$; in particular, $BG$ always denotes the quotient $[X/G]$ as a stack over $X$.

Throughout the paper $\X$ is an algebraic stack over $X$; unless otherwise indicated, the map $\X \to X = \spec A$ is a good moduli space morphism and is of finite type. We denote by $\X_{\liset}$ the lisse-\'etale topos of $\X$ and define $\hhat\O_\X$ to be the completion $\varprojlim \O_\X/\I^n$, where $\I$ is the sheaf of ideals generated by the pullback of $\m\subseteq A$. Following \cite[\S1]{BC:gaga}, we define the ringed topos $\hhatX$ to be the pair $(\X_{\liset},\hhat\O_\X)$.
There is a natural completion functor
\[
 \Coh(\X)\to \Coh(\hhatX),\qquad
 \F \mapsto \hhat\F := \varprojlim \F/\I^{n+1}\F.
\]
Letting $\X_n=\X\times_{\spec A}\spec A/\m^{n+1}$, the natural functor $\Coh(\hhatX)\to \varprojlim \Coh(\X_n)$ is an equivalence of categories \cite[Theorem 2.3]{BC:gaga}, where the map $\Coh(\X_n)\to \Coh(\X_{n-1})$ is given by pullback along the closed immersion $\X_{n-1}\to \X_n$. We may therefore regard elements of $\Coh(\hhatX)$ as compatible systems of coherent sheaves on the $\X_n$.

\section{Proof of Theorem \ref{T:main}}\label{S:proof-of-main}

This section is quite technically involved. Subsequent sections depend on the results but not on the techniques or terminology developed in this section.
\medskip

We use the terminology of topoi developed in \cite{SGA4}. A morphism of topoi $f\colon Y\to X$ is a triple $(f_*,f^{-1},\alpha)$, where $f^{-1}\colon X \to Y$ is a functor which commutes with finite limits, $f_*\colon Y \to X$ is a functor, and $\alpha$ is an adjunction $\hom_Y(f^{-1}(-),-) \xrightarrow{\sim} \hom_X(-,f_*(-))$. If $\O_Y$ and $\O_X$ are sheaves of rings on $Y$ and $X$,
respectively, then a morphism of ringed topoi (also denoted $f\colon
Y\to X$) is a morphism of topoi, together with a morphism of sheaves of rings
$f^{-1}\O_X\to \O_Y$. In this case, $f_*\colon \O_Y\mod\to\O_X\mod$ is right adjoint to $f^*(-)=f^{-1}(-)\otimes_{f^{-1}\O_X}\O_Y$.

\begin{definition}
 A morphism of ringed topoi $f\colon Y\to X$ is \emph{flat} if $f^*$ is exact.
\end{definition}


\begin{lemma}\label{L:flat=>Hom_isom}
  If $f\colon Y\to X$ is a flat morphism of ringed topoi, $\F$ is a locally finitely presented $\O_X$-module, and $\G$ is any $\O_X$-module, then the natural map $f^*\Hom_{\O_X}(\F,\G)\to \Hom_{\O_Y}(f^*\F,f^*\G)$ is an isomorphism.
\end{lemma}
\begin{proof}
  \underline{Case 1}: If $\F\cong \O_X$, the natural map is isomorphic to the identity map on $f^*\G$. Similarly, if $\F\cong \O_X^{\oplus I}$, the map is isomorphic to the canonical isomorphism $f^*(\G^{\oplus I}) \to (f^*\G)^{\oplus I}$.

  \underline{Case 2}: Suppose $\F$ has a global presentation
  \[
    \O_X^{\oplus J}\to \O_X^{\oplus I}\to \F\to 0.
  \]
  Since $f^*$ is right exact, we get a global presentation
  \[
    \O_Y^{\oplus J}\to \O_Y^{\oplus I}\to f^*\F\to 0.
  \]
  Applying $\Hom_{\O_X}(-,\G)$ to the first sequence and
  $\Hom_{\O_Y}(-,f^*\G)$ to the second, we get the exact sequences
  \[\xymatrix@R-1.5pc @C-1.3pc{
    0\ar[r] & \Hom_{\O_X}(\F,\G) \ar[r] & \Hom_{\O_X}(\O_X^{\oplus I},\G) \ar[r] & \Hom_{\O_X}(\O_X^{\oplus J},\G) \\
    0\ar[r] & \Hom_{\O_Y}(f^*\F,f^*\G) \ar[r] &
    \Hom_{\O_Y}(\O_Y^{\oplus I},f^*\G) \ar[r] &
    \Hom_{\O_Y}(\O_Y^{\oplus J},f^*\G).
  }\]
  Since $f$ is flat, the first sequence remains exact if we apply $f^*$, so the rows in the following diagram are exact. The squares commute by naturality of the vertical arrows.
  \[\xymatrix@C-1.3pc{
    0\ar[r] & f^*\Hom_{\O_X}(\F,\G) \ar[r]\ar[d] & f^*\Hom_{\O_X}(\O_X^{\oplus I},\G) \ar[r]\ar[d]^\wr & f^*\Hom_{\O_X}(\O_X^{\oplus J},\G)\ar[d]^\wr \\
    0\ar[r] & \Hom_{\O_Y}(f^*\F,f^*\G) \ar[r] &
    \Hom_{\O_Y}(\O_Y^{\oplus I},f^*\G) \ar[r] &
    \Hom_{\O_Y}(\O_Y^{\oplus J},f^*\G).
  }\]
  We have already shown that the middle and right vertical arrows are isomorphisms, so the left vertical arrow must also be an isomorphism, completing the proof in the case where $\F$ is globally presented.

  \underline{Case 3}: Now we prove the general case. To check that the natural map $f^*\Hom_{\O_X}(\F,\G)\to\Hom_{\O_Y}(f^*\F,f^*\G)$ is an isomorphism, it is enough to find a cover of the final object of $Y$ so that it pulls back to an isomorphism. Since $\F$ is quasi-coherent, there is a cover of the final object of $X$ so that the pullback of $\F$ has a presentation. Pulling that cover back along $f$, we get a cover of the final object of $Y$ (here we're using exactness of $f^{-1}$ to say that the final object pulls back to the final object and that covers pull back to covers on canonical sites). On that cover, the map is an isomorphism by case 2. The construction of $\Hom$, the application of $f^*$, and the construction of the natural map are local on $X$, so the natural morphism constructed on the cover is the restriction of the natural morphism on $Y$.
\end{proof}

\begin{lemma}\label{L:completion_is_flat}
  If $\X$ is a Noetherian algebraic stack and $\I\subseteq \O_\X$ is a quasi-coherent sheaf of ideals, then $\hhat \O_\X$, the completion of $\O_\X$ with respect to $\I$, is flat over $\O_\X$. That is, the canonical map $\iota\colon\hhatX\to\X$ is a flat morphism of ringed topoi.
\end{lemma}
\begin{proof}
  Let $\F \to \F'$ be an injection of $\O_{\X}$-modules. We
  need to check injectivity of the map
  \[
    \F \otimes_{\O_{\X}} \hhat \O_{\X} \to \F' \otimes_{\O_{\X}} \hhat \O_{\X}.
  \]
   Since sheafification is exact, it suffices to check injectivity of the maps
  \[
    \F(U) \otimes_{\O_{\X}(U)} \hhat \O_{\X}(U) \to \F'(U) \otimes_{\O_{\X}(U)} \hhat \O_{\X}(U)
  \]
  as $U$ varies over a base for $\X_{\liset}$. Thus it suffices
  to check that the above maps are injections for $f\colon U \to \X$ a
  smooth map and $U$ an affine scheme. By definition, $\O_{\X}(U) = \O_U(U)$ and $\hhat\O_\X(U)=\hhat\O_U(U)$. Since $U$ is affine, $\hhat\O_U(U) = \hhat {\O_U(U)}$. Injectivity follows since $\hhat {\O_U(U)}$ is flat over $\O_U(U)$ \cite[Theorem 7.2b]{Eisenbud}.
\end{proof}

\begin{remark}\label{R:pullback=completion}
  The same trick of restricting to affine schemes smooth over $\X$ shows that for any coherent sheaf $\F$ on $\X$, the natural map $\iota^*\F\to \hhat\F$ is an isomorphism. (Note however that this is not true for \emph{quasi-}coherent sheaves.)
\end{remark}

\begin{remark}\label{R:completion-is-exact}
  Lemma \ref{L:completion_is_flat} and Remark \ref{R:pullback=completion} show that completion of coherent sheaves is exact.
\end{remark}

\begin{lemma}\label{L:vectorUnicity}
 Suppose $\phi\colon\X\to \spec A$ is a good moduli space, where $A$ is a complete Noetherian local ring with maximal ideal $\m$. Additionally assume that $\phi$ has cohomological dimension 0. Then any vector bundle $\V$ on $\X_{n-1}$ is the reduction of a unique vector bundle on $\X_n$. In particular, any vector bundle on $\X_0$ extends to a unique vector bundle on $\hhatX$.
\end{lemma}

\begin{remark}[Cohomological dimension of cohomologically affine morphisms]
\label{R:cohomological-dimension-error}
If $\X$ has affine diagonal  then $\phi$ has cohomological dimension 0 -- i.e.~ $R^i\phi_* = 0$ for $i>0$.  Indeed, by \cite[Remark 3.5]{Alper:good}), cohomologically affine stacks with non-affine diagonal are \emph{not} cohomologically of dimension 0. The reason is that the morphism of triangulated categories
\[
D^+(\qcoh(X)) \to   D^+_{\qcoh}(\O_X\mod)
\]
is not an isomorphism unless $\X$ has affine diagonal, and the derived functors are computed in the second category. (See e.g. \cite[Tag 07B5]{stacks-project}.)

 An easy example is that for an elliptic curve $E$ over a field $k$,  $f\colon \spec k \to BE$ is not cohomologically of dimension zero; this follows from pulling back by the smooth cover $f\colon \spec k \to BE$ and cohomology and base change. A more essential counterexample is the structure morphism $g \colon BE \to \spec k$, the hypercover spectral sequence associated to  $f\colon \spec k \to BE$ gives $H^1(BE,\O_{BE}) \neq 0$.

\end{remark}

\begin{proof}[Proof of Lemma \ref{L:vectorUnicity}]
  This is a direct application of \cite[Theorem 8.5.3(b)]{fga}. The obstruction to extending $\V$ to $\X_n$ lies in $H^2(\X_{n-1},\I^n\otimes \mathcal{E}nd(\V))$, which vanishes since $\X_{n-1}$ is cohomologically of dimension 0. Therefore $\V$ extends. Moreover, the isomorphism classes of extensions are parameterized by $H^1(\X_{n-1},\I^n\otimes \mathcal{E}nd(\V))$, which vanishes by the same argument, so the extension is unique.
\end{proof}

\begin{lemma}\label{L:flatness-and-tensoring}
Suppose $\phi\colon\X\to \spec A$ is a good moduli space, where $A$ is a complete Noetherian local ring with maximal ideal $\m$.   Then a quasi-coherent sheaf $\F$ on a locally Noetherian stack $\X$ is a flat $\O_\X$-module (i.e.~restricts to a flat sheaf on any smooth cover by a scheme) if and only if $\F\otimes_{\O_\X}-$ is an exact functor on $\qcoh(\X)$.
\end{lemma}
\begin{proof}
  Suppose $\F$ is flat and $\G\to \G'$ is an injection of quasi-coherent sheaves. Let $f\colon U\to \X$ be a smooth cover by a scheme. We may check that $\F\otimes \G\to \F\otimes \G'$ is injective after pulling back to $U$. Pullback respects tensor products, $f^*\G\to f^*\G'$ is injective (since $f$ is flat), and $f^*\F$ is a flat $\O_U$-module, so $f^*(\F\otimes \G)\to f^*(\F\otimes \G')$ is injective.

  For the converse, again let $f\colon U\to \X$ be a smooth cover by a scheme. We wish to prove that $f^*\F$ is flat. This may be done locally on $U$, so we may assume $U$ is a Noetherian affine scheme. The result is well-known for schemes, so it suffices to prove that $f^*\F\otimes_{\O_U}-$ is an exact functor on $\qcoh(U)$. First we claim that for any $\O_U$-module $\G$, the counit of adjunction $f^*f_*\G \to \G$ has a natural section. Indeed, let $W \to U$ be a smooth morphism. Then the map
  \[
    f^*f_*\G(W \to U) \cong \G(U\times_{\X}W \to U) \to \G(W \to U)
  \]
  has a section given by the restriction map 
  \[
    \G(W \to U) \to \G(U\times_{\X}W \to W \to U) .
  \]
  Now let $\G \to \G'$ be an injection of quasi-coherent sheaves on $U$. Since $U$ is Noetherian, $f$ is quasi-compact and quasi-separated, so $f_*\G \to f_*\G'$ is an injection of quasi-coherent $\O_\X$-modules. By assumption, 
  $\F\otimes_{\O_{\X}} f_*\G \to \F\otimes_{\O_{\X}}f_*\G'$ is an injection, and since $f$ is flat, $\phi\colon f^*(\F \otimes_{\O_{\X}}f_*\G) \to f^*(\F\otimes_{\O_{\X}}f_*\G')$ is an injection. Noting that $f^*(\F\otimes_{\O_{\X}}f_*(-)) \cong f^*\F\otimes_{\O_{U}}f^*f_*(-)$, we  get a  diagram
  \[\xymatrix{
    f^*\F\otimes_{\O_{U}} f^*f_*\G \ar[r]^\phi \ar[d]^\pi
    & f^*\F\otimes_{\O_{U}} f^*f_*\G' \ar[d]
    \\
    f^*\F\otimes_{\O_{U}}\G \ar[r]^\psi  \ar@/^/[u]^\sigma
    & f^*\F\otimes_{\O_{U}}\G' \ar@/^/[u]^{\sigma'}
  }\]
  We have that $\phi$ is injective, and $\sigma$ is injective (since it is a section of $\pi$). Since $\sigma'\circ\psi=\phi\circ\sigma$, we conclude that $\psi$ is injective.
\end{proof}

\notocsubsection{Proof of Theorem \ref{T:main}}
  Part (\ref{item:full-faithfulness}):
  For any coherent $\O_\X$-modules $\F$ and $\G$, we must show that the natural map $\hom_{\O_\X}(\F,\G)\to \hom_{\hhat \O_\X}(\hhat\F,\hhat G)$ is an isomorphism. We have that $\Hom(\F,\G)$ is coherent. By Lemma \ref{L:flat=>Hom_isom}, Lemma \ref{L:completion_is_flat}, and Remark \ref{R:pullback=completion}, the natural map $\hhat{\Hom_{\O_\X}(\F,\G)}\to \Hom_{\hhat\O_\X}(\hhat\F,\hhat \G)$ is an isomorphism. By \cite[Proposition 4.7 (iii)]{Alper:good}, the induced map on global sections is the desired isomorphism.

  \smallskip
  \notocsubsection{\myref{res}$\Rightarrow$\myref{res$'$}} This is immediate since any coherent sheaf on $\X_0$ is a coherent sheaf on $\X$.
  
  \notocsubsection{\myref{res$'$}$\Rightarrow$\myref{GAGA}} By part (\ref{item:full-faithfulness}), the completion functor is fully faithful. It remains to show that any compatible system $\fF=\{\F_n\}_{n\ge 0}$ of coherent sheaves on the stacks $\X_n$ is induced by a coherent sheaf $\F$ on $\X$. As usual, we denote by $\I$ the quasi-coherent sheaf of ideals generated by $\phi^*(\m)$.

By \myref{res$'$}, there exist a locally free sheaf $\V$ on $\X$ and a surjection $\V\to \F_0$. We inductively argue that for each $n$ this lifts to a surjection $\V \to \F_n$. The bottom row of the following diagram is exact:
  \[\raisebox{1.5pc}{$\xymatrix{
    & & & \V \ar@{->>}[d] \ar@{.>}[dl] \\
    0\ar[r]& \I^n \F_n\ar[r] & \F_n\ar[r] & \F_n/\I^n \F_n \ar[r] &
    0
  }$}\tag{$\ddag$}\]
Since $\V$ is a vector bundle, the following sequence is exact:
  \[
  0\to \Hom_{\O_\X}(\V,\I^n \F_n)\to \Hom_{\O_\X}(\V,\F_n)\to \Hom_{\O_\X}(\V,\F_n/\I^n\F_n)\to 0.
  \]
  By cohomological affineness of $\phi$, the sequence remains exact when we take global sections, so the composition map $\hom_{\O_\X}(\V,\F_n)\to \hom_{\O_\X}(\V,\F_n/\I^n\F_n)$ is surjective. Thus, there is a lift $\V\to \F_n$ as indicated by the dotted arrow in $(\ddag)$. The induced map $\V \to \F_n$ surjective by Nakayama's lemma. This gives a compatible system of maps $\{\V\to \F_m\}_{m\ge 0}$, and thus a surjective morphism $\widehat \V\to \fF$.

  Repeating the above argument for the kernel of $\widehat \V\to \fF$, we get a presentation $\widehat \W\to \widehat \V\to \fF\to 0$, where $\V$ and $\W$ are vector bundles on $\X$. By part (\ref{item:full-faithfulness}), the morphism $\widehat \W\to \widehat\V$ is induced by some $\O_\X$-module homomorphism $\W\to \V$. Let $\G$ be the cokernel of this map. By Remark \ref{R:completion-is-exact}, the top row of the following diagram is exact.
  \[\xymatrix@R-1pc{
    \widehat \W\ar[r]\ar@{=}[d] & \widehat \V\ar[r]\ar@{=}[d] & \widehat \G \ar@{.>}[d]\ar[r] & 0\\
    \widehat \W\ar[r] & \widehat \V\ar[r] & \fF\ar[r] & 0
  }\]
  The induced morphism from $\widehat \G$ to $\fF$ is therefore an isomorphism.

  \notocsubsection{\myref{res$'$}$\Rightarrow$\myref{res}}
  The above argument shows that if \myref{res$'$} holds and $\F$ is a coherent sheaf on $\X$, then there is a vector bundle $\V$ on $\X$ and a surjection $\hhat\V\to \hhat\F$. Since \myref{res$'$}$\Rightarrow$\myref{GAGA}, this map is induced by a surjection $\V\to \F$.

  \notocsubsection{\myref{GAGA}$\Rightarrow$\myref{res$'$}} First we show that if $\X$ has affine diagonal and if \myref{GAGA} holds, any $\F\in \Coh(\X)$ whose completion is a vector bundle on $\hhatX$ is a vector bundle.  By Remark \ref{R:pullback=completion}, the equivalence of categories of coherent sheaves respects tensor products, so since $\hhat\F\otimes_{\hhat\O_\X}-$ is an exact functor on $\Coh(\hhatX)$, we have that $\F\otimes_{\O_\X}-$ is an exact functor on $\Coh(\X)$. Let $\spec R\to \X$ be a smooth cover of $\X$ (note $\X$ is assumed of finite type over the Noetherian ring $A$, so it is quasi-compact).
  Then $\spec R\times_\X\spec R$ is of finite type over $\spec A$, so the projections $\spec R\times_\X\spec R\to \spec R$ are smooth, quasi-compact, and quasi-separated, so any quasi-coherent sheaf on $\X$ is the limit of its coherent subsheaves \cite[Lemma \href{http://stacks.math.columbia.edu/tag/07TU}{07TU}]{stacks-project}.
  Since $\F\otimes_{\O_\X}-$ commutes with direct limits, it is exact on the category of \emph{quasi}-coherent sheaves, so $\F$ is a flat $\O_\X$-module by Lemma \ref{L:flatness-and-tensoring}. It follows that $\F$ is a vector bundle; indeed, this can be checked smooth locally, and a flat coherent sheaf on a Noetherian affine scheme is locally free \cite[Theorem 2.9 of Chapter 1]{Milne:etaleBook}.

  Now for any coherent sheaf $\F_0$ on $\X_0$, since $\X_0$ is assumed to have the resolution property, there is a vector bundle $\V_0$ on $\X_0$ with a surjection to $\F_0$. By Lemma \ref{L:vectorUnicity}, $\V_0$ extends to a vector bundle on $\hhatX$, which by formal GAGA is the completion of a coherent sheaf $\V$ on $\X$. By the above paragraph, $\V$ is a vector bundle. Now $\V\to \V_0\to \F_0$ is a surjection. This shows that \myref{res$'$} holds.
  
\phantomsection\label{quot<=>quot'=>res}
\notocsubsection{\myref{quot}$\Leftrightarrow$\myref{quot$'$}}
It is clear that \myref{quot}$\Rightarrow$\myref{quot$'$}. Conversely, suppose
  $\X=[V/G]$ for some algebraic space $V$ and some subgroup
  $G\subseteq GL_n$. Let $U=(V\times GL_n)/G$, where $g\cdot (v,h)=(v\cdot g^{-1},g\cdot h)$ (alternatively, $U$ is the pullback of the  universal $GL_n$-torsor along the composition $[V/G]\to BG\to BGL_n$). Then $\X=[U/GL_n]$. 
  
Since $U\to \X$ is a $GL_n$-torsor, it is an
  affine morphism, and $\X\to \spec A$ is cohomologically affine, so
  $U\to \spec A$ is cohomologically affine. As $U$ has trivial stabilizers, it is an algebraic space, so by Serre's criterion
  \cite[Theorem III.2.5]{Knutson}, $U$ is an affine scheme.

\notocsubsection{\myref{quot}$\Rightarrow$\myref{res}$\Rightarrow$\myref{quot$'$}}
By \cite[Corollary 5.9]{gross:generators-arxiv} \myref{quot}$\Rightarrow$\myref{res}, and if the closed point of $\X$ has affine stabilizer, then by
\cite[Lemma 4.1]{Totaro} \myref{res}$\Rightarrow$\myref{quot$'$}.
\hfill$\square$

\begin{remark}\label{R:affine-gms+quot=>res}
  The proofs of \myref{quot}$\Leftrightarrow$\myref{quot$'$} and \myref{quot}$\Rightarrow$\myref{res}$\Rightarrow$\myref{quot$'$} apply to \emph{any} stack with affine good moduli space. Note however that \myref{res}$\Rightarrow$\myref{quot$'$} requires all closed points of the stack to have affine stabilizer.
\end{remark}

\begin{remark}
  Note that the proof of \myref{GAGA}$\Rightarrow$\myref{res$'$} shows that \myref{GAGA} implies that any coherent sheaf whose completion is a vector bundle must be a vector bundle. The hypothesis that $\X_0$ have the resolution property is not necessary for this result.
\end{remark}

\begin{remark}\label{R:X-constant=>GAGA}
  Suppose $A/\m=k$ and $A$ is a $k$-algebra (this is automatic if $k$ has characteristic zero\footnote{Every non-negative integer is non-zero in $k$, so lies in $A\smallsetminus \m$, so is invertible in $A$. This shows that $A$ is a $\mathbb Q$-algebra. By \cite[Theorem 7.7]{Eisenbud}, it is a $k$-algebra.}).
  If $\X\cong \X_0\times_{\spec k}\spec A$, then we have a morphism $s\colon \X\to \X_0$ so that $\X_0\hookrightarrow \X\xrightarrow s \X_0$ is the identity map. Any vector bundle $\V_0\in\Coh(\X_0)$ is the reduction of the vector bundle $s^*\V_0\in \Coh(\X)$. If $\X_0$ has the resolution property, then any $\F_0\in\Coh(\X_0)$ has a surjection from a vector bundle $\V_0\in\Coh(\X_0)$, so the map $s^*\V_0\to\V_0\to\F_0$ is a surjection from a vector bundle on $\X$. That is, if $\X_0$ has the resolution property, \myref{res$'$} holds.

  Note however that the condition $\X\cong \X_0\times_{\spec k}\spec
  A$ is frequently not satisfied. For example, consider the
  $j$-invariant map $j\colon \mathcal M_{1,1} \to \AA^1_{\CC}$
and let $\X \to \spec \CC\llbracket t\rrbracket$ be the pullback of $j$ to the local ring of $\AA^1_{\CC}$ at the origin.
  Since elliptic curves with $j$-invariant 0 have automorphism group $\ZZ/4\ZZ$ but generic elliptic curves have automorphism group $\ZZ/2\ZZ$, $\X$ cannot be the pullback of its special fiber.
\end{remark}

\section{Formal GAGA is finite flat local (and \'etale local) on the base}
\label{S:etaleLocal}

\begin{lemma}\label{L:etale-local}
  Suppose $\phi\colon \X\to \spec A$ is a good moduli space, where $A$ is a complete Noetherian local ring and $\phi$ is of finite type. Suppose $\spec A'\to \spec A$ is a finite flat morphism, where $A' \neq 0$ is again local (and therefore a complete Noetherian local ring), and let $\X'=\X\times_{\spec A}\spec A'$.

  If $\Coh(\X')\to \Coh(\hhatX')$ is essentially surjective, then so is $\Coh(\X)\to \Coh(\hhatX)$.
\end{lemma}

\begin{remark}\label{R:finite-local=>etale-local}
  We note that any morphism of spectra of complete Noetherian local rings $\spec A'\to \spec A$ which is an \'etale cover must be finite flat. Since such a morphism is surjective, $A'/\m_A A'$ is some \'etale (and so finite) extension of $A/\m_A$. In particular, $\m_A A' = \m_{A'}$. By \cite[Proposition 18.3.2]{EGAIV.4}, there is a \emph{finite} \'etale morphism $\spec B\to \spec A$ inducing the same extension of $A/\m_A$. By the formal criterion for \'etaleness and the fact that $A'$ and $B$ are each complete with respect to $\m_A$, there are unique morphisms $\spec A'\to \spec B$ and $\spec B\to \spec A'$ over $\spec A$ lifting the isomorphism of extensions of $A/\m_A$, and these must be inverses. Thus, $\spec A'\to \spec A$ is finite flat.
\end{remark}

\begin{remark}\label{R:pullback-and-completion}
  By Remark \ref{R:pullback=completion}, completion of coherent sheaves agrees with pullback along the morphism of topoi $\iota\colon \hhatX\to \X$. It follows that pullback along $\pi\colon \X'\to \X$ commutes with completion of coherent sheaves, and that completion of coherent
  sheaves is a right exact functor. To see this, note that the hypotheses of $A' \neq 0$ and finiteness imply that $\m_{A'}^n \subseteq \m_A A' \subseteq \m_{A'}$ for some $n$. Indeed, the second containment can only fail if $\m_A$ contains a unit of $A'$, in which case $\m_A A' = A'$, so Nakayama's lemma implies $A'=0$. For the first containment, we can reduce to the case $\m_A=0$, so $A$ is a field. Then $A'$ is a finite-dimensional vector space, so $\m_{A'}^n$ stabilizes for large $n$, and it must stabilize to 0 (again by Nakayama). We can therefore regard both $\hhatX$ and $\hhatX'$ as completions with respect to the pullback of $\m_A$.

\end{remark}

\begin{proof}[Proof of Lemma \ref{L:etale-local}]
  Good moduli space morphisms are stable under base change \cite[Proposition 4.7(i)]{Alper:good} and composition, so $\X''=\X'\times_\X\X' \to \spec A'' = \spec(A'\otimes_A A')$ is a good moduli space. Let $p_1,p_2\colon \X''\to\X'$ denote the projections. While $A''$ may no longer be a local ring, $\spec
  A''$ is finite flat over $\spec A'$, so it must be a disjoint
  union $\bigsqcup \spec A''_i$, where each $A''_i$ is a
  complete local ring. Let $\X''_i =\X'' \times_{\spec A''} \spec A''_i$.

  Let $\fF\in \Coh(\hhatX)$, and let $\fF'\in \Coh(\hhatX')$ be the pullback to $\hhatX'$. By assumption, $\fF'$ is the completion of a sheaf $\F'\in \Coh(\X')$. Applying Theorem \ref{T:main}(\ref{item:full-faithfulness}) to each of the good moduli space morphisms
  $\X''_i \to \spec A''_i$, we see that the descent datum $p_2^*\fF'\xrightarrow\sim p_1^*\fF'$ is induced by a map $p_2^*\F'\xrightarrow\sim p_2^*\F'$ (note we are using Remark \ref{R:pullback-and-completion}). By fppf descent for coherent sheaves, $\F'$ is the pullback of a coherent sheaf $\F$ on $\X$. Since $\hhat \F$ and $\fF$ are defined by the same descent datum, they are isomorphic.
\end{proof}

\begin{theorem}[Formal GAGA is finite flat and \'etale local on the base]\label{T:etale-local}
  In the setup of Lemma \ref{L:etale-local}, formal GAGA holds for $\X\to\spec A$ if and only if it holds for $\X'\to\spec A'$.
\end{theorem}

\begin{proof}
  By Theorem \ref{T:main}(\ref{item:full-faithfulness}), both completion functors are fully faithful.

  By Lemma \ref{L:etale-local}, if the completion functor $\Coh(\X')\to \Coh(\hhatX')$ is essentially surjective, then so is $\Coh(\X)\to \Coh(\hhatX)$.

  Conversely, suppose $\Coh(\X)\to \Coh(\hhatX)$ is essentially surjective, and let
  $\fF\in\Coh(\hhatX')$. Since $\pi\colon \X' \to \X$ is finite, $\pi_*\fF\in\Coh(\hhatX)$. By assumption, $\pi_*\fF\cong \hhat\F$ for some $\F\in\Coh(\X)$. The
  composition $\pi^*\hhat\F  \xrightarrow\sim \pi^*\pi_* \fF \to \fF$ is
  a surjection. Let $\fG$ denote the kernel of this map. By the same argument, there exists a surjection $\pi^*\hhat \G \to \fG$ for some $\G\in\Coh(\X)$. Then $\fF$ is the cokernel of the map $\pi^*\hhat\G\to \pi^*\hhat\F$. By full faithfulness and Remark \ref{R:pullback-and-completion}, this map is induced by a morphism $\pi^*\G\to \pi^*\F$, and the cokernel of this map has completion $\fF$.
\end{proof}




\begin{remark}
  In the same spirit, we note that the resolution property also descends along finite flat morphisms \cite[Proposition 4.3 (vii)]{gross:generators-arxiv}.
\end{remark}

\section{Counterexamples to formal GAGA}
\label{S:counterexamples}

Recall that for a relative group scheme $G\to S$, a coherent sheaf on $BG=[S/G]$ is equivalent to a coherent sheaf on $S$ with a $G$-linearization (i.e.~a $G$-action). Pushforward along $\phi\colon BG\to S$ corresponds to taking the subsheaf of invariants; in particular, since $\O_{BG}$ corresponds to $\O_S$ with the trivial $G$-action, $\phi$ is Stein. Since the action of $G$ on $S$ is trivial, $\phi$ is universal for maps to algebraic spaces.\footnote{More generally, if $\alpha\colon G\times X\to X$ is an action of $G$ on an algebraic space $X$ and the two maps $\alpha, p_2\colon G\times X\to X$ have coequalizer $Y$ in the category of algebraic spaces, then $[X/G]\to Y$ is universal for maps to algebraic spaces.} The condition that the map be cohomologically affine is precisely the condition that $G$ is linearly reductive. Therefore $BG\to S$ is a good moduli space if and only if $G$ is linearly reductive. (We note that this holds even if $G$ is not affine; while in the usual definition of reductive of \cite{SGA3}, $G \to S$ is affine, we use the notion of \emph{linearly reductive} of \cite[Definition 12.1]{Alper:good}, which is a cohomological condition.)

Formal GAGA fails without the good moduli space condition. In the
following, we say that a morphism to an algebraic space $\X \to X$ is
a \emph{no-good moduli space} if it is universal for maps to algebraic
spaces but is not a good moduli space.

\begin{example}[Counterexample to full faithfulness for a no-good moduli space]
  Let $A=k\llbracket t\rrbracket$ for a field $k$ of characteristic not 2. Let $G=\spec k\llbracket t\rrbracket \sqcup \spec k\llparen t\rrparen$, regarded as an open subgroup of $(\ZZ/2\ZZ)_{\spec A}$. Then $\X=BG\to \spec A$ is not a good moduli space. The non-trivial 1-dimensional representation of $\ZZ/2\ZZ$ induces a non-trivial rank 1 vector bundle on $\X$ whose completion is the trivial rank 1 vector bundle on $\hhatX$ (indeed, $\hhatX\cong \spec A$), showing that the completion functor is not fully faithful.
\end{example}

\begin{example}[Counterexample to essential surjectivity for a no-good moduli space]
  Formal GAGA fails for $B\GG_a$. For a ring $R$, a line bundle on $B\GG_{a,R}$ is equivalent to a 1-dimensional representation of $\GG_{a,R}$ (i.e.~a group homomorphism $\GG_{a,R} \to\GG_{m,R}$). The formula $x \mapsto \exp(tx) =\sum_{i = 0}^{\infty} \frac{t^i}{i!}x^i$ gives a compatible family of homomorphisms $\GG_{a,\CC[t]/t^n} \to\GG_{m,\CC[t]/t^n}$ which do not lift to a homomorphism $\GG_{a,\CC\llbracket t\rrbracket} \to\GG_{m,\CC\llbracket t\rrbracket}$.
\end{example}

Formal GAGA may also fail for good moduli spaces.

\begin{example}[Counterexample to essential surjectivity with non-separated diagonal]\label{Eg:B(non-sep)}
  Let $A=k\llbracket t\rrbracket$ for a field $k$. Let $G$ be $\spec k\llbracket t\rrbracket$ with a doubled origin, regarded as a group over $\spec A$. Since $G$ is a quotient of $(\ZZ/2\ZZ)_{\spec A}$ by a flat subgroup scheme it is linearly reductive by \cite[Proposition 12.17]{Alper:good}, so $\X=BG\to \spec A$ is a good moduli space.

  Any vector bundle on $\X$ consists of a vector bundle $\V$ on $\spec A$ and a group homomorphism $G_A\to \mathrm{Aut}_A(\V)$. Since $\mathrm{Aut}_A(\V)$ is separated, such a map must factor through the trivial group. So any vector bundle on $\X$ corresponds to a vector bundle on $\spec A$ with trivial $G$-action.
  However, $\hhatX\cong B_{\spec A}(\ZZ/2\ZZ)$, so there are formal vector bundles not of this form, namely those induced by non-trivial representations of $\ZZ/2\ZZ$.
\end{example}

Even if we require separated diagonal, formal GAGA may still fail.

\begin{example}[Counterexample to essential surjectivity with separated, non-affine diagonal]
\label{ex:separatedNonAffineDiagonal}
  Let
  \[
    G' = \proj \bigl(k\llbracket t\rrbracket[x,y,z]/(zy^2-x^2(x+z) - tz^3)\bigr)
  \]
  where $t$ has degree 0 and $x$, $y$, and $z$ have degree 1. Let $G$ be the complement of the origin of
  the special fiber, with structure map $\pi\colon G\to \spec
  k\llbracket t\rrbracket$. The generic fiber is an elliptic curve
  $E\to \spec k\llparen t\rrparen$, but the special fiber is isomorphic to $\GG_m$.
By \cite[IV Theorem 5.3(c)]{Silverman:Advanced}, $G$ is a relative   group scheme over $\spec k\llbracket t\rrbracket$. We claim that $BG\to \spec k\llbracket t\rrbracket$ is a good moduli space morphism (i.e.~that taking $G$-invariants is exact on $G$-linearized coherent sheaves).

  To see this, we first note that any deformation of the group scheme $\GG_m$ is
  trivial. By \cite[Expos\'e III, Corollaire 3.9]{SGA3}, isomorphism classes of deformations of the group scheme along a square-zero ideal $I$ (if they exist) are parameterized by
  $H^{2}(\GG_m,\Lie(\GG_m)\otimes I)$, where $\Lie(\GG_m)$ is the adjoint representation and $I$ has the trivial action. The group cohomology $H^i(\GG_m,-)$ as defined in \cite[Expos\'e III, 1.1]{SGA3} is simply the \v{C}ech cohomology associated to the cover $\spec k\llbracket t\rrbracket \to B\GG_m$. Since $\GG_m$ is affine, this \v{C}ech cohomology agrees with sheaf cohomology on $B\GG_m$. Since $\GG_m$ is linearly reductive, $B\GG_m\to \spec k\llbracket t\rrbracket$ is cohomologically affine, so the higher cohomology groups vanish. Thus, the only deformation of $\GG_m$ is $\GG_m$.

  Next, any torsion $G$-linearized coherent sheaf is supported over $\spec(k[t]/t^n)$ for some $n$. That is, there is some choice of $n$ so that the given sheaf is in the essential image of $j_*$ in the diagram below.
  \[\xymatrix@C-1pc{
    (B\GG_m)_{\spec(k[t]/t^n)} \ar@{}[r]|-{\mbox{$\cong$}} \ar[d]_{\pi_n} & BG\times_{\spec k\llbracket t\rrbracket}\spec(k[t]/t^n) \ar@{^(->}[r]^-j & BG\ar[d]^\pi\\
    \spec(k[t]/t^n) \ar@{^(->}[rr]^i & & \spec k\llbracket t\rrbracket
  }\]
  Since $i$ and $j$ are affine, and $\pi_n$ is cohomologically affine, we have
  \[
    R\pi_*\circ j_* = R(\pi_*\circ j_*) = R(i_*\circ \pi_{n*})=i_*\circ \pi_{n*} = \pi_*\circ j_*.
  \]
  That is, torsion sheaves on $BG$ have trivial higher cohomology.

  Any torsion-free $G$-linearized coherent sheaf is free with trivial action. Indeed,
  it is free with some rank $r$ since $k\llbracket t\rrbracket$ is a DVR. The action of $G$ is given by some group homomorphism $G\to GL_{r,k\llbracket t\rrbracket}$. Since $G$ has proper connected generic fiber and $GL_r$ is affine, this map must be trivial over the generic point. Since $G$ is reduced and $GL_r$ is separated, the map must be trivial.

  Any $G$-linearized coherent sheaf $\F$ (with torsion subsheaf $\F^{\tor}$) fits into a $G$-equivariant short exact sequence
  \begin{equation}
   0\to \F^{\tor} \to \F\to \F/\F^{\tor}\to 0.\tag{$\ast$}
  \end{equation}
  Since $\F/\F^\tor$ is free, the following sequence is exact.
  \[
   0\to \Hom(\F/\F^\tor,\F^\tor)\to \Hom(\F/\F^\tor,\F)\to \Hom(\F/\F^\tor,\F/\F^\tor)\to 0.
  \]
  Since $\Hom(\F/\F^\tor,\F^\tor)$ is torsion, $H^1(BG,\Hom(\F/\F^\tor,\F^{\tor})) = 0$, so the sequence remains exact when we take global sections. Global sections of $\Hom(\F,\G)$ are $G$-equivariant maps from $\F$ to $\G$, so there is a $G$-equivariant splitting of the sequence ($\ast$). We have shown that any $G$-linearized coherent $k\llbracket t\rrbracket$-module $M$ decomposes into a direct sum of its torsion part $M^{\tor}$ (with trivial cohomology) and a free part $M^{\free}$ (with trivial action).

  Suppose we have a short exact sequence of linearized
  modules $0\to M''\to M\xrightarrow\phi M'\to 0$.
  We wish to show that any invariant $m'\in M'$ is the
  image of an invariant element of $M$. Since $\phi$ is surjective, we have that $m'=\phi(m_f+m_t)$, where $m_t$ is torsion and $m_f$ is invariant. Since torsion sheaves have trivial
  cohomology, any invariant torsion element which is the image of a
  torsion element is actually the image of an \emph{invariant} torsion element, so $\phi(m_t)=\phi(n_t)$ for some invariant torsion element $n_t\in M$. Then $m_f+n_t$ is invariant and $\phi(m_f+n_t)=m'$. This completes the proof that $BG\to\spec k\llbracket t\rrbracket$ is a good moduli space morphism.

  \medskip
  Now take  any vector bundle over the origin with non-trivial $\GG_m$
  action. By Lemma \ref{L:vectorUnicity}, this extends to a
  unique vector bundle on $\hhat{BG}$, but we have seen that there is no
  torsion-free coherent sheaf on $BG$ with non-trivial action on the special fiber.
\end{example}

\begin{remark}\label{R:localQuotient-conj-is-false}
  A similar example gives a counterexample to \cite[Conjecture 1]{alper:localQuotient}. Let
  \[
    G' = \proj \bigl(\CC[t,x,y,z]/(zy^2-x^2(x+z) - tz^3)\bigr)
  \]
  where $t$ has degree 0 and $x$, $y$, and $z$ have degree 1. Let $G$ be the largest subscheme of $G'$ over which the map to $\AA^1=\spec \CC[t]$ is smooth. By \cite[IV Theorem 5.3(c)]{Silverman:Advanced}, $G$ is a relative group over $\AA^1$. Let $\X=BG$. This $\X$ is finitely presented over $\spec \CC$, and the image under the quotient map $\AA^1 \to \X$ of the origin of $\AA^1$ is a closed point $x$ with stabilizer $\GG_m$. If \cite[Conjecture 1]{alper:localQuotient} were true, there would be an algebraic space $Y$ with a point $y$ and an \'etale representable morphism $f\colon [Y/\GG_m]\to \X$ sending $y$ to $x$, inducing an isomorphism of stabilizers. As $f$ is \'etale, its image is open, so the image contains some closed point of $\X$ whose stabilizer is an elliptic curve. By \cite[Proposition 3.2]{toric2}, $f$ induces finite-index inclusions of stabilizers. But no subgroup of $\GG_m$ can possibly be a finite-index subgroup of an elliptic curve.
  
  It is possible that \cite[Conjecture 1]{alper:localQuotient} holds for stacks with affine stabilizers.
\end{remark}

\begin{remark}
  Taking $\X=BG$ and $\X'=B\GG_m$ over $A=k\llbracket t\rrbracket$, Example \ref{ex:separatedNonAffineDiagonal} shows that the natural map
  \[
    \hom_A(\X,\X') \to \hom_A(\hhatX,\hhatX')
  \]
  is not necessarily an equivalence of categories. The complex analytic analogue of this natural map is an equivalence of categories if $\X$ is a proper Deligne-Mumford
  stack and $\X'$ is either a quasi-compact algebraic stack with
  affine diagonal \cite[Theorem 1.1]{lurie:tannaka} or a locally of finite type
  Deligne-Mumford stack with quasi-compact and quasi-separated diagonal \cite[Theorem 1]{hall:GAGA}.
\end{remark}

\begin{remark}\label{R:affine-diag=>GAGA}
  It is difficult to imagine an example of a stack $\X$ with affine
  diagonal and good moduli space $\spec A$ which is not a quotient
  stack (i.e.~does not satisfy \myref{quot$'$}) \'etale locally on $\spec A$. Likely candidates, such as non-trivial $\GG_m$-gerbes, do not work (see Remark \ref{R:gerbes}). If
  no such stack exists, then Theorems \ref{T:main} and \ref{T:etale-local} show that formal GAGA holds provided that $\X$ has affine diagonal.
\end{remark}
\begin{conjecture}\label{conjecture}
  Suppose $\phi\colon \X\to \spec A$ is a good moduli space morphism, where $A$ is a complete Noetherian local ring and $\phi$ is of finite type. If $\X$ has affine diagonal, then the completion functor $\Coh(\X)\to \Coh(\hhatX)$ is an equivalence of categories. (Note: a forthcoming result implies this conjecture holds; see Remark \ref{R:AHR}.)
\end{conjecture}

 Formal GAGA may hold even if $\X$ does not have affine diagonal, but it is usually uninteresting. For example, for any elliptic curve $E \to \spec A$, formal GAGA holds for $BE \to \spec A$ since all coherent sheaves on $BE$ are pulled back from $\spec A$.

\section{Application to the local quotient structure of good moduli spaces}
\label{potentialApplication}

Recall that a stack $\X$ is a \emph{quotient stack} if it is
the stack quotient of an algebraic space by a subgroup of $GL_n$ for some $n$ (i.e.~if \myref{quot$'$} holds).

\begin{proposition}\label{P:localQuotient}
 Let $\phi\colon \X \to X$ be a stack over $X$ with affine diagonal,  with $\phi$ of finite type and $X$ a locally Noetherian scheme. Assume that $\phi$ is a good moduli space. Let $x \in X$ be a point such that the fiber $\X_0$ over $x$ is a quotient stack. Suppose that formal GAGA holds for
 $\ttilde\X = \X \times_X \spec \hhat \O_{X,x} \to \spec \hhat \O_{X,x}$. 
 Then there exists an \'etale neighborhood $X' \to X$ of $x$ such that $\X\times_X X'$ is a quotient stack.
\end{proposition}

\begin{remark}
To apply the proposition to the case where $X$ is an algebraic space and $x$ is a topological point, one would find an \'etale neighborhood $U\to X$ with $U$ a scheme and a point $u\in U$ which maps to $x$, then apply the proposition to $\X\times_X U\to U$. Any two \'etale covers have a common refinement, so by Theorem \ref{T:etale-local}, the formal GAGA hypothesis is satisfied for one \'etale cover if and only if it is satisfied for any \'etale cover. (Also see Remark \ref{R:GAGA-use-in-localQuotient}.)
\end{remark}




\begin{proof}
  The question is \'etale local on $X$, and by Theorem \ref{T:etale-local} the hypothesis is \'etale local on $X$, so we may assume that $X = \spec R$ is an
  affine scheme. Let $X^{\h} = \spec R^{\h}$, where $R^{\h}$ is the strict henselization
  of $R$ at $x$, and let $X^{\loc} = \spec{\hhat{R^{\h}}}$. Let $\X^{\loc}$ and $\X^{\h}$ denote the pullback of $\X$ to $X^{\loc}$ and $X^{\h}$, respectively. For a sheaf $\F$ on $\X$ (or $\X^h$), let $\F^{\h}$ and $\F^\loc$ denote the pullback of $\F$ to $\X^{\h}$ and $\X^\loc$, respectively. Let $\hhatX$ be the completion of $\X^\loc$ with respect to the maximal ideal of $\hhat{R^h}$.

\[\xymatrix@R-1.4pc{
  \hhatX \ar[r] \ar[dd] & \X^{\loc}\ar[r]\ar[dd] & \X^{\h}\ar[r]\ar[dd] & \X' \ar[r]\ar[dd] & \X\ar[dd]\\ \\
  X^{\loc} \ar@{}[r]|{\mbox{$=$}} & X^{\loc}\ar@{}[d]|{\parallel}\ar[r] & X^h\ar@{}[d]|{\parallel}\ar[r] & X_j\ar@{}[d]|{\parallel}\ar[r] & X\\
   & \spec\hhat{R^{\h}}& \spec R^h& \spec R_j
}\]
  
  The closed substack $\X_0\subseteq \ttilde\X$ is a quotient stack, so its unique closed point has affine stabilizer, and it has the resolution property by Remark \ref{R:affine-gms+quot=>res}. By assumption, \myref{GAGA} holds for $\ttilde\X\to \spec\hhat\O_{X,x}$, so by Theorem \ref{T:main}(\ref{item:existence}), $\ttilde\X=[U/GL_n]$ for an affine scheme $U$. Since $\spec\hhat{R^h}\to\spec\hhat\O_{X,x}$ is an affine morphism, $\X^\loc\to \ttilde\X$ is affine, so $U^\loc=U\times_{\ttilde\X}\X^\loc$ is an affine scheme and $\X^\loc=[U^\loc/GL_n]$. By Remark \ref{R:affine-gms+quot=>res}, $\X^\loc$ has the resolution property.

  Next we show that $\X^{h}$ has the resolution property.
  Let $\F$ be a coherent sheaf on $\X^{\h}$. By the previous paragraph, there is a vector bundle $\V^{\loc}$ on $\X^{\loc}$ with a surjection to $\F^{\loc}$.
  By \cite[Proposition 4.18(i)]{LaumonMB:champs}, the stack of rank $n$ vector bundles on $\X$, $\Hom(\X, BGL_n)$, is locally of finite presentation over $X$.
  By Artin approximation \cite[Theorem 1.12]{artin:approximation}, there exists a vector bundle
  $\V$ on $\X^{\h}$ such that the pullback of  $\V$ to $\X_0$ is the
  same as the pullback of $\V^{\loc}$ to $\X_0$. By Lemma
  \ref{L:vectorUnicity} and Theorem
  \ref{T:main}(\ref{item:full-faithfulness}), the pullback of $\V$ to
  $\X^{\loc}$ is isomorphic to $\V^{\loc}$. Since $\Hom(\V,\F)$ is
  locally of finite presentation and the substack of surjections is
  open in $\Hom(\V,\F)$ \cite[Lemma 2.2.2]{lieblich:stackOfCoherent}, the substack of surjections is locally of finite presentation.
  By Artin approximation, there exists a surjection $\V \to \F$. This proves that $\X^h$ has the resolution property.

  Let $\X^h_0$ denote the closed fiber of $\X^h$. The morphism $\X^h_0\to \X_0$ is a representable morphism to a quotient stack. If $\X_0=[U/G]$, then $\X^h_0=[(U\times_{\X_0}\X^h_0)/G]$, so $\X^h_0$ is a quotient stack. In particular, the closed point of $\X^h$ has affine stabilizer, so by Remark \ref{R:affine-gms+quot=>res}, $\X^{\h}=[P^h/GL_n]$ for some affine scheme $P^h$. The $GL_n$-torsor $P^h\to \X^h$ corresponds to a representable map $p^h\colon \X^{\h} \to BGL_n$.
  Since $\Hom(\X, BGL_n)$ is locally of finite presentation over $X^h$ and since $R^{\h} = \varinjlim R_i$, where the limit runs over all \'etale neighborhoods $X_i=\spec R_i\to \spec R$ of $x$, we have that $p^h$ is the pullback of some $p_i\colon \X_i=\X\times_X X_i\to BGL_n$.
  Let $\Q_i \to \X_i$ be the corresponding $GL_n$-torsor.
  To finish the proof, it suffices to show there exists an \'etale neighborhood $X_j\to X_i$ such that $\Q_j=\Q_i\times_{X_i}X_j$ is an affine scheme.

  Since $X$ is locally Noetherian, $X^h$ is Noetherian \cite[Proposition 18.8.8(iv)]{EGAIV.4}. As $P^h$ is of finite type over $X^h$, it is finitely presented over $X^h$, so there exists an \'etale neighborhood $X_{j_0} \to X_i$ and an affine scheme $P_{j_0}$ over $X_{j_0}$ such that $P^h \cong P_{j_0} \times_{X_{j_0}} X^{\h}$.
  Let $\Q_{j_0} = \Q_i \times_{X_i}X_{j_0}$. By \cite[Proposition 4.18(i)]{LaumonMB:champs},
$\Hom_{X_{j_0}}(\Q_{j_0},P_{j_0})$ and $\Hom_{X_{j_0}}(P_{j_0},\Q_{j_0})$ are locally of finite presentation over $X_{j_0}$, so there exists an \'etale neighborhood $X_{j_1}  \to X_{j_0}$ such that the isomorphism $f\colon \Q_i\times_X X^\h = \Q_{j_0}\times_{X_{j_0}} X^{\h} \to P_{j_0}\times_{X_{j_0}} X^{\h} = P^h$ and its inverse $g$ are the pullbacks
of maps $f_1$ and $g_1$ which are defined over $X_{j_1}$. By \cite[Proposition 4.18(i)]{LaumonMB:champs}, there is an \'etale neighborhood $X_{j}  \to X_{j_1}$ such that the compositions $f_1 \circ g_1$ and $g_1 \circ f_1$ pull back to the identities over $X_j$. This shows that $\Q_{j}\cong P_{j}$ is an affine scheme, as desired.
\end{proof}

\begin{remark}\label{R:GAGA-use-in-localQuotient}
  In the proof of Proposition \ref{P:localQuotient}, the formal GAGA
  hypothesis is only used to show that $\X^{\loc}$ has the resolution
  property. If this
  can be obtained in some other way (e.g.~if formal GAGA holds
  for $\X^{loc} \to \spec \hhat{\O_{X,x}^{\h}}$), the rest of this proof works as
  above.
\end{remark}


Because of results like Proposition \ref{P:localQuotient}, and more generally because of the strategy presented in Remark \ref{R:template}, it is desirable to have a classification of stacks which have a point as a good moduli space. It is not known which such stacks are quotient stacks.

\begin{question}\label{Q:non-quotient_with_one_point_GMS}
  Does there exist a good moduli space morphism $\Y \to \spec k$, with $k$ a
  separably closed field, such that $\Y$ has affine diagonal 
 but is not a quotient stack? (Note: a forthcoming result resolves this question in the negative; see Remark \ref{R:AHR}.)
\end{question}


\begin{remark}
\label{R:gerbes}
  One natural source of
  examples is non-trivial gerbes. By \cite[Example 3.12]{EHKV} there
  are $\GG_m$-gerbes which are not quotient stacks. If $\X$ is a
  $\GG_m$-gerbe over a Noetherian scheme $X$, then $\X$ is a quotient
  stack if and only if its class in $H^2(X,\GG_m)$ is in the image of
  the Brauer map $\Br(X) \to H^2(X,\GG_m)$  \cite[Theorem 3.6]{EHKV}. However, if $X = \spec A$, where $A$ is a complete local ring (e.g.~a field), then by
  \cite[Corollary IV.2.12]{Milne:etaleBook},\footnote{Note that in
    contrast to \cite{EHKV},  Milne defines $\Br'(X) = H^2(X,\GG_m)$ (p.~147).}
  the natural map  $\Br(X) \to H^2(X,\GG_m)$ is an isomorphism, so any $\GG_m$-gerbe over $\spec A$ is a quotient stack.
\end{remark}

\begin{remark}
  Another candidate counterexample is $\fM^{\leq m}_0$, the moduli stack of genus 0 prestable (i.e.~nodal) curves with at most $m$
  nodes. Over any field $K$ and for any $m\ge 2$, $\fM^{\leq m}_0$ is not a quotient stack \cite[Proposition 5.2]{kresch:flatteningPreStable}. The map $\fM^{\leq m}_0 \to \spec K$ is universal for maps to algebraic spaces. However, it is not a good moduli space map since closed points of $\fM^{\leq m}_0$ can
  have non-reductive stabilizers: each outer leaf of any tree $T$ of
  smooth rational curves contributes a copy of $\Aut \AA^1 \cong \GG_m
  \ltimes \GG_a$ to $\Aut T$. But by \cite[Proposition
  12.14]{Alper:good}, the stabilizers at closed points of a stack which has a good moduli space are linearly reductive.

  A promising variant is $\fM^{\leq m}_{0,n}$, the moduli stack of marked genus 0
  prestable curves with $n$ marked points and at most $m$ nodes, such that each component has at least two marks/nodes. The closed points of this stack have linearly reductive stabilizers, and the stack is non-empty for $n\ge 2$. For $m\ge 2$, we sketch a modification of Kresch's argument to show that $\fM^{\le m}_{0,n}$ is not a quotient stack. 
There is an open immersion $\fM^{\le 2}_{0,n}\subseteq \fM^{\le m}_{0,n}$, so it suffices to show $\fM^{\le 2}_{0,n}$ is not a quotient stack.
  There is a representable morphism $\ttilde\fM^{\le 2}_{0,n}\to \fM^{\le 2}_{0,n}$ from the stack in which the points are \emph{labeled}, so it suffices to check that the former is not a quotient stack. There is a morphism $\ttilde\fM^{\le 2}_{0,2}\to \ttilde\fM^{\le 2}_{0,n}$ given by adding points in a prescribed fashion, which is a trivial $\GG_m$-gerbe over its image (for $n\ge 3$), so it suffices to check $\ttilde\fM^{\le 2}_{0,2}$ is not a quotient stack. A straightforward modification of the proof of \cite[Proposition 5.2]{
kresch:flatteningPreStable} shows that $\ttilde\fM^{\le 2}_{0,2}$ is not a quotient stack.

  For $n\ge 3$ and $m\ge 2$, there are are curves which isotrivially degenerate to multiple closed points, so $\fM^{\le m}_{0,n}$ cannot have a good moduli space by \cite[Proposition 4.16(iii)]{Alper:good}. The stack $\fM^{\le 2}_{0,2}$ has a unique closed point (topologically, it is a chain of 3 points) and the map to a point is universal for maps to algebraic spaces. If this map were a good moduli space morphism, it would answer Question \ref{Q:non-quotient_with_one_point_GMS} affirmatively.
\end{remark}

\begin{remark}\label{R:Jarod-conj=>question}
  Suppose $\X\to \spec A$ is a good moduli space as in \S\ref{S:terminology}, with $k=A/\m$ separably closed. Suppose $\X$ has affine diagonal, and satisfies \cite[Conjecture 1]{alper:localQuotient} (by Remark \ref{R:localQuotient-conj-is-false}, we cannot expect this unless $\X$ has affine stabilizers). Let $G_x$ be the stabilizer of the unique closed point $x$ of $\X$. Then there is a representable \'etale morphism $f\colon \W=[U/G_x]\to \X$ and a point $w \in \W(k)$ such that the induced map $\Aut_{\W(k)}(w) \to \Aut_{\X(k)}(x)=G_x$ is an isomorphism. Suppose the strong form of this conjecture holds (i.e.~ that we may take  $U = \spec R$ to be affine; see \cite[second paragraph after Conjecture 1]{alper:localQuotient}).
  
  (This argument was suggested to us by Jarod Alper.) Let $\W=[\spec R/G_x]\to \X$ be as above. By \cite[Theorem 5.1]{Alper:good}, the induced map on good moduli spaces $\spec R^{G_x}\to \spec A$ is \'etale. Since $A$ is complete with separably closed residue field, the component of $\spec R^{G_x}$ containing the image of $w$ must be isomorphic to $\spec A$, so after shrinking $\spec R$, we may assume $f\colon \W\to \X$ induces an isomorphism of good moduli spaces.

  We claim that $f$ is an isomorphism. Since $f$ is \'etale, its image is open. Any open set containing the unique closed point $x$ of $\X$ is all of $\X$, so $f$ is an \'etale cover. We may check that a morphism is an isomorphism \'etale locally on the base, so it suffices to show that the projection $p_1\colon \W\times_\X\W\to \W$ is an isomorphism. By \cite[Proposition 4.7(i)]{Alper:good} $\W\times_\X\W$ has good moduli space $\spec A\times_{\spec A}\spec A=\spec A$, so it has a unique closed point. The diagonal $\W\to \W\times_\X\W$ has this closed point in its image. As the diagonal is a section of an \'etale morphism, it is an open immersion, so it is an isomorphism.

  The strong form of \cite[Conjecture 1]{alper:localQuotient} for stacks with affine diagonal therefore answers Question \ref{Q:non-quotient_with_one_point_GMS} negatively: if $A=k$ is a separably closed field, the above argument shows that $\X$ is a quotient stack.
\end{remark}

\begin{remark}\label{R:Jarod-conj=>conjecture}
  If the strong form of \cite[Conjecture 1]{alper:localQuotient} for stacks with affine diagonal is true, the following argument shows that Conjecture \ref{conjecture} is true. In this case, the formal GAGA hypothesis in Proposition \ref{P:localQuotient} may be replaced by the hypothesis that $\X$ has affine diagonal.

  Let $G_x$ denote the stabilizer of the closed point $x$ of $\X$. By \cite[Proposition 12.14]{Alper:good}, $G_x$ is linearly reductive. Let $\W=[\spec R/G_x]\to \X$ and $w\in \W$ be as in \cite[Conjecture 1]{alper:localQuotient}. By \cite[Theorem 5.1]{Alper:good}, the map on good moduli spaces $\spec R^{G_x}\to \spec A$ is \'etale, so after shrinking $\spec R$, we may assume $\spec R^{G_x}\to \spec A$ is a finite \'etale extension. As $\W$ and $\X\times_{\spec A}\spec R^{G_x}$ are both \'etale over $\X$, the induced morphism $\W\to \X\times_{\spec A}\spec R^{G_x}$ is \'etale. This morphism induces an isomorphism on good moduli spaces, and the image contains the unique closed point of $\X\times_{\spec A}\spec R^{G_x}$.
  By the argument in Remark \ref{R:Jarod-conj=>question}, the map is an isomorphism. As $\W$ is a quotient stack, formal GAGA holds for $\W\to \spec R^{G_x}$ by Theorem \ref{T:main}(\ref{item:existence}). By Theorem \ref{T:etale-local}, formal GAGA holds for $\X\to \spec A$.
\end{remark}


\newcommand{\etalchar}[1]{$^{#1}$}

\end{document}